\documentclass[12pt]{article}
\usepackage{fullpage, amsmath, amssymb, amsfonts, amsthm, stmaryrd, url,graphicx,float,multirow,caption}
\usepackage[all]{xy}
\usepackage{tikz-cd}
\usepackage{float}
\usepackage{tikz}
\usetikzlibrary{trees}
\usetikzlibrary{decorations.pathmorphing}
\usetikzlibrary{decorations.pathreplacing}
\usetikzlibrary{decorations.markings}
\usetikzlibrary{decorations.shapes}
\usepackage{verbatim}
\usepackage[utf8x]{inputenc}
\usepackage{ucs}
\usepackage{enumerate}
\usepackage{mathrsfs}
\usepackage{bm}
\usepackage{pgfplots}
\usetikzlibrary{patterns}
\usetikzlibrary{matrix,arrows}
\usepackage[colorlinks=true]{hyperref}
\usepackage[nottoc,numbib]{tocbibind}
\usepackage{ifthen}
\usepackage{everypage}
\usepackage{mathtools}
\usepackage{url}
\usepackage{hyperref}
\usepackage{utopia}
\hypersetup{urlcolor=blue,linkcolor=blue,citecolor=blue,colorlinks=true}

\newtheorem{theorem}{Theorem}[section]
\newtheorem{lemma}[theorem]{Lemma}

\newtheorem{prop}[theorem]{Proposition}
\theoremstyle{definition}
\newtheorem{defn}[theorem]{Definition}

\newtheorem{remark}[theorem]{Remark}

\numberwithin{equation}{theorem}

\newcommand{\CC}{\mathbb{C}}

\newcommand{\QQ}{\mathbb{Q}}
\newcommand{\A}{\bm{\mathrm{A}}}

\newcommand{\ZZ}{\mathbb{Z}}

\newcommand{\Gm}{\mathbb{G}_m}

\DeclareMathOperator{\Br}{Br}

\DeclareMathOperator{\Gal}{Gal}

\DeclareMathOperator{\Hom}{Hom}

\DeclareMathOperator{\im}{Im}

\DeclareMathOperator{\Pic}{Pic}

\DeclareMathOperator{\Spec}{Spec}

\DeclareMathOperator{\SL}{SL}

\DeclareMathOperator{\Ker}{Ker}

\usepackage{fancyhdr}
\usepackage{geometry}   
\usepackage[T2A,T1]{fontenc}
\makeatletter
\def\easycyrsymbol#1{\mathord{\mathchoice
  {\mbox{\fontsize\tf@size\z@\usefont{T2A}{\rmdefault}{m}{n}#1}}
  {\mbox{\fontsize\tf@size\z@\usefont{T2A}{\rmdefault}{m}{n}#1}}
  {\mbox{\fontsize\sf@size\z@\usefont{T2A}{\rmdefault}{m}{n}#1}}
  {\mbox{\fontsize\ssf@size\z@\usefont{T2A}{\rmdefault}{m}{n}#1}}
}}
\makeatother
\newcommand{\Be}{\easycyrsymbol{\CYRB}}
\newcommand{\Sha}{\easycyrsymbol{\CYRSH}}

\title{Weak approximation for homogeneous spaces over some two-dimensional geometric global fields}
\author{Haowen Zhang}

\begin{document}

\maketitle
\begin{abstract}
    In this article, we study obstructions to weak approximation for connected linear groups and homogeneous spaces with connected or abelian stabilizers over finite extensions of $\CC((x,y))$ or function fields of curves over $\CC((t))$. We show that for connected linear groups, the usual Brauer-Manin obstruction works as in the case of tori. However, this Brauer-Manin obstruction is not enough for homogeneous spaces, as shown by the examples we give. We then construct an obstruction using torsors under quasi-trivial tori that explains the failure of weak-approximation. 
\end{abstract}
\section{Introduction}
Given an algebraic variety $Z$ over a number field $K$, we say that $Z$ satisfies weak approximation if $Z(K)$ is dense in $\prod_{v\in{\Omega}}Z(K_v)$ with respect to the product of the $v$-adic topologies, where $K_v$ denotes the $v$-adic completion and $\Omega$ denotes the set of places of $K$. Weak approximation is satisfied for simply connected groups, and when it is not satisfied, we want to characterize the  closure $\overline{Z(K)}$ of the set of rational points in $\prod_{v\in{\Omega}}Z(K_v)$ as some subset of elements satisfying certain compatibility conditions, for example, as the Brauer-Manin set $Z(K_{\Omega})^{\Br}$ (c.f. Section 8.2 of \cite{poonen2017rational}), and this gives an obstruction to weak approximation. 
\par
Going beyond number fields, recently there's been an increasing interest in studying analogous questions over some two-dimensional geometric global fields.  As in \cite{izquierdo2021local}, we consider a field $K$ of one of the following two types:
\begin{enumerate}[(a)]
    \item the function field of a smooth projective curve $C$ over $\CC((t))$;
    \item the fraction field of a local, henselian, two-dimensional, excellent domain $A$ with algebraically closed residue field of characteristic 0 (e.g. any finite extension of $\CC((x,y))$, the field of Laurent series in two variables over the field of complex numbers).
\end{enumerate}
In case (a), one can consider the set $\Omega$ of valuations coming from the closed points of the curve $C$. Colliot-Thélène/Harari (c.f. \cite{Colliot_Th_l_ne_2015}) proved the following exact sequence describing the obstruction to weak approximation for the case of tori $T$:
\begin{equation}\label{main}
1\rightarrow\overline{T(K)}\rightarrow\prod_{v\in\Omega}T(K_v)\rightarrow\Be_\omega(T)^D\rightarrow\Be(T)^D\rightarrow 1
\end{equation}
where $\Be(Z)$ (resp. $\Be_\omega(Z)$) is defined to be the subgroup of $\Br_1(Z)/\Br(K)$ containing elements vanishing in $\Br_1(Z_v)/\Br(K_v)$ for all places (resp. almost all places) $v\in C^{(1)}$, and $A^{D}$ denotes the group of continuous homomorphisms from $A$ to $\QQ/\ZZ(-1)$. \par
In case (b), one can take $\Omega$ to be the set of valuations coming from prime ideals of height one in $A$. Izquierdo (c.f. \cite{izquierdo2019dualite}) proved the exact sequence \ref{main} in such situations for tori $T$.
\par Using some dévissage arguments as in \cite{Colliot_Th_l_ne_2015}, we can first generalize this result to a connected linear group $G$ over $K$.
\begin{theorem}(Theorem \ref{4.1})
We keep the notation as above. For a connected linear group $G$ over $K$ of the type (a) or (b), there is an exact sequence
\begin{equation}
1\rightarrow\overline{G(K)}\rightarrow G(K_\Omega)\rightarrow\Be_\omega(G)^D\rightarrow\Be(G)^D\rightarrow 1.
\end{equation}
\end{theorem}
In particular, the obstruction to weak approximation is controlled by the Brauer set $G(K_\Omega)^{\Be_\omega}:=\Ker(G(K_\Omega)\rightarrow\Be_\omega(G)^D)$. However, such an exact sequence doesn't generalize to homogeneous spaces, and we do find counter-examples:
\begin{theorem}(Proposition \ref{thm2})
Let $Q$ be a flasque torus such that
$$H^1(K,Q)\rightarrow\oplus_{v\in\Omega}H^1(K_v,Q)\rightarrow\Sha^1_{\omega}(\widehat Q)^D$$ is not exact (for example the one constructed in Corollary 9.16 of \cite{Colliot_Th_l_ne_2015}). 
Embed $Q$ into some $\SL_n$ and let $Z:=\SL_n/Q$. 
Then $\overline {Z(K)}\subsetneq Z(K_\Omega)^{\Be_\omega}.$
\end{theorem}
In order to better understand the obstruction to weak approximation for homogeneous spaces over fields of the type (a) or (b), we should somehow combine the Brauer-Manin obstruction with the descent obstruction, another natural tool used in the study of such questions, as done by Izquierdo and Lucchini Arteche in \cite{izquierdo2021local} for the study of obstruction to rational points. 
\begin{theorem}(Proposition \ref{thm3})
For $G$ a connected linear group, we consider a homogeneous space $Z$ under $G$ with geometric stabilizer $\bar H$ such that $G^{\text{ss}}$ is simply connected and $\bar H^{\text{torf}}$ is abelian. Define 
$$Z(K_\Omega)^{\text{qt},\Be_\omega}:=\bigcap_{f:W\xrightarrow T{Z},T\text{ quasi-trivial}}f(W(K_\Omega)^{\Be_\omega}).$$
where $f$ runs over torsors $W\rightarrow Z$ under quasi-trivial tori $T$. Then $Z(K_\Omega)^{\text{qt},\Be_\omega}=\overline{Z(K)}.$
\end{theorem}
\section{Acknowledgments}
The author would like to thank Cyril Demarche for his enormous help as a PhD supervisor and a great number of useful discussions, without which it wouldn't be possible to see this article coming out.
\section{Notation and preliminaries}\label{notation}
The notation will be fixed in this section and used throughout this article. The setting is pretty much the same as in \cite{izquierdo2021local} where they treated the problem of Hasse principle for such varieties.\par
\paragraph{Cohomology} The cohomology groups we consider are always in terms of Galois cohomology or étale cohomology.
\paragraph{Fields}
Throughout this article, we consider a field $K$ of one of the following two types:
\begin{enumerate}[(a)]
    \item the function field of a smooth projective curve $C$ over $\CC((t))$;
    \item the field of fractions of a two-dimensional, excellent, henselian, local domain $A$ with algebraically closed residue field of characteristic zero. An example of such a field is any finite extension of the field of fractions $\CC((X, Y))$ of the formal power series ring $\CC[[X,Y]]$.
\end{enumerate}
Both these two types of fields share a number of properties that hold also for totally imaginary number fields:
\begin{enumerate}[(i)]
    \item the cohomological dimension is two,
    \item index and exponent of central simple algebras coincide,
    \item for any semisimple simply connected group $G$ over $K$, we have $H^1(K,G)=1$.
    \item there is a natural set $\Omega$ of rank one discrete valuations (i.e. with values in $\ZZ$), with respect to which one can take completions. For type (a), let $\Omega$ be the set of valuations coming from $C^{(1)}$ the set of closed points of the curve $C$. For type (b), let $\Omega$ be the set of valuations coming from prime ideals of height one in $A$.
    
\end{enumerate}
Weak approximation is satisfied for semisimple simply connected groups over such fields (c.f. Théorème 4.7 of \cite{colliot2004arithmetic} for type (b), and \S 10.1 of \cite{Colliot_Th_l_ne_2015} for type (a)).
\paragraph{Sheaves and  abelian groups}
For $i>0$, we denote by $\mu_n^{\otimes i}$ the $i$-time tensor product of the étale sheaf $\mu_n$ of $n$-th roots of unity with itself. We set $\mu_n^0=\ZZ/n\ZZ$ and for $i<0$, and we define $\mu_n^{\otimes i}=\Hom(\mu_n^{\otimes(-i)},\ZZ/n\ZZ).$ Over the fields $K$ we consider, since $K$ contains an algebraically closed field and thus all the roots of unity, we have a (non-canonical) isomorphism $\mu_n\simeq\ZZ/n\ZZ$. Denote by $\QQ/\ZZ(i)$ the direct limit of the sheaves $\mu_m^{\otimes i}$ for all $m>0$. By choosing a compatible system of primitive $n$-th roots of unity for every $n$ (for example $\xi_n=\exp(2\pi i/n)$), we can identify $\QQ/\ZZ(i)$ with $\QQ/\ZZ$.
\par
For an abelian group $A$ (that we always suppose to be equipped with the discrete topology if there isn't any other topology defined on it), we denote by $A^D$ the group of continuous homomorphisms from $A$ to $\QQ/\ZZ(-1)$. The functor $A\mapsto A^D$ is an anti-equivalence of categories between torsion abelian groups and profinite groups.

\paragraph{Weak approximation and Brauer-Manin obstruction}
Given a set $\Omega$ of places of $K$, let $Z(K_\Omega):=Z(K_\Omega)$ be the topological product where each $Z(K_v)$ is equipped with the $v$-adic topology, and $K_v$ denotes the completion at $v$. We say that a $K$-variety $Z$ satisfies weak approximation with respect to $\Omega$ if the set of rational points $Z(K)$, considered through the diagonal map as a subset of $Z(K_\Omega)$, is a dense subset.
Weak approximation is a birational invariant, which is a consequence of the implicit function theorem for $K_v$ (c.f. Theorem 9.5.1 of \cite{colliot2021brauer} and the argument in the Proposition 12.2.3 of \cite{colliot2021brauer}).
\par
There are varieties for which weak approximation fails, thus we try to introduce obstructions that give a more precise description of the closure of $Z(K)$ inside $Z(K_\Omega))$, thus explaining such failures. 
\par 
For a variety $Z$, the cohomological Brauer group $\Br(Z)$ is defined to be $H^2(Z,\Gm).$
Define also
\begin{itemize}
    \item $\Br_0(Z):=\im(\Br(K)\rightarrow\Br(Z))$,
    \item $\Br_1(Z):=\ker(\Br(Z)\rightarrow\ Br(Z_{\bar K}))$,
     \item $\Br_{\text{a}}(Z):=\Br_1(Z)/\Br_0(Z),$
    \item $\Be(Z):=\ker(\Br_{\text{a}}(Z)\rightarrow\prod_{v\in\Omega}\Br_{\text{a}}(Z_{K_{v}}))$,
       \item $\Be_S(Z):=\ker(\Br_{\text{a}}(Z)\rightarrow\prod_{v\in S}\Br_{\text{a}}(Z_{K_{v}}))$ where $S$ is a finite subset of $\Omega$,
   
    \item $\Be_\omega(Z)$: subgroup of $\Br_{\text{a}}(Z)$ containing the elements vanishing in $\Br_{\text{a}}(Z_{K_v})$ for almost all places $v\in\Omega$. 
\end{itemize}
Similar to the exact sequence from Class Field Theory for number fields, there is also an exact sequence (c.f. Prop. 2.1.(v) of \cite{Colliot_Th_l_ne_2015} and Thm. 1.6 of \cite{izquierdo2019dualite})
$$\Br(K)\rightarrow\bigoplus_{v\in \Omega}\Br(K_v)
\rightarrow \QQ/\ZZ\rightarrow 0$$
and a Brauer-Manin pairing
\begin{align*}
    Z(K_\Omega)\times\Be_\omega&\rightarrow \QQ/\ZZ\\
    ((P_v),\alpha)& \mapsto \sum_{v\in\Omega}<P_v,\alpha>
    \end{align*}
with the following property: $Z(K)$ is contained in the subset $Z(K_\Omega)^{\Be_\omega}$ of $Z(K_\Omega)$ defined as $\{(P_v)\in Z(K_\Omega):((P_v),\alpha)=0\text{ for all }\alpha\in\Be_{\omega}(Z)\}$: the set of points that are orthogonal to $\Be_{\omega}(Z)$. When $Z(K)$ is non-empty and $\overline{Z(K)}=Z(K_\Omega)^{\Be_\omega}$, we say that the Brauer-Manin obstruction with respect to $\Be_\omega$ is the only one to weak approximation.
\paragraph{Tate-Shafarevich groups} For a Galois module $M$ over the field $K$ and an integer $i\geq 0$, we define the following Tate-Shafarevich groups:
\begin{itemize}
    \item $\Sha^i(K,M):=\ker(H^i(K,M)\rightarrow\prod_{v\in\Omega}H^i(K_v,M))$,
        \item $\Sha_S^i(K,M):=\ker(H^i(K,M)\rightarrow\prod_{v\in S}H^i(K_v,M))$ where $S$ is a finite subset of $\Omega$,
    \item $\Sha_{\omega}^1(K,M):$ subgroup of $H^i(K,M)$ containing the elements vanishing in $H^i(K_v,M)$ for almost all places $v\in\Omega$.
\end{itemize}

\paragraph{Algebraic groups and homogeneous spaces}
For a linear algebraic $K$-group $G$, the following notation will be used:
\begin{itemize}
\item $D(G)$: the derived subgroup of $G$,
\item $G^{\circ}$: the neutral connected component of $G$,
\item $G^{\text{f}}:=G/G^\circ$ the group of connected components of $G$, which is a finite group,
\item $G^{\text u}:$ the unipotent radical of $G^\circ$,
\item $G^{\text{red}}:=G^\circ/G^\text u$ which is a reductive group,
\item $G^{\text{ss}}:=D(G^{\text{red}})$ which is a semisimple group,
\item $G^{\text{tor}}:=G^{\text{red}}/G^{\text{ss}}$ which is a torus,
\item $G^{\text{ssu}}:=\ker(G^{\circ}\rightarrow G^{\text{tor}})$ which is an extension of $G^{\text{ss}}$ by $G^\text u$,
\item $G^{\text{torf}}:=G/G^{\text{ssu}}$ which is an extension of $G^\text{f}$ by $G^{\text{tor}}$,
\item $\hat G$: the Galois module of the geometric characters of $G$.
\end{itemize}
\par
A unipotent group over $K$ a field of characteristic $0$ is isomorphic to an affine space $\A_K^n$, thus $K$-rational and satisfies weak approximation.\par
A torus $T$ is said to be quasi-trivial if $\hat T$ is an induced $\Gal(\bar K/K)$-module.

\section{Weak approximation for connected linear groups}
The aim of this section is to prove the following theorem which concerns the Brauer-Manin obstruction to weak approximation for connected linear groups:
\begin{theorem}\label{4.1}
Let $K$ be a field of the type (a) or (b) and let $G$ be a connected linear group over $K$. Then there is an exact sequence
$$1\rightarrow\overline{G(K)}\rightarrow G(K_\Omega)\rightarrow\Be_\omega(G)^D\rightarrow\Be(G)^D\rightarrow 1.$$
\end{theorem}
We complete the proof by a series of lemmas.
\begin{lemma} Suppose that Theorem 4.1 holds for $G^n$, then it also holds for $G$.
\end{lemma}
\begin{proof}
It follows from the fact that all of the operations in the exact sequence above take products to products. (For example, the closure of a product is the product of closures.)
\end{proof}
An exact sequence $$1\rightarrow F\rightarrow H\times_K P\rightarrow G\rightarrow 1$$ where $H$ is semi-simple simply connected, $P$ is a quasitrivial $K$-torus and $F$ is finite and central is called a special covering of the reductive $K$-group $G$. For any reductive $K$-group $G$, there exists an integer $n>0$ such that $G^n$ admit a special covering. By the lemma above, we can suppose without loss of generality that $G$ admits a special covering itself.

\begin{lemma}
The following sequence is exact
$$1\rightarrow \overline{G(K)}\rightarrow\prod_{v\in \Omega}G(K_v)\rightarrow \Be_\omega(G)^D$$
where $G$ is a reductive connected linear group.
\end{lemma}
\begin{proof}
The proof is exactly the same as in Lemma 9.6 of \cite{Colliot_Th_l_ne_2015}, using the special covering and the fact that $H$ and $P$ both satisfy weak approximation (so does their product) and $H^1(K,H)=H^1(K,P)=H^1(K,H\times_K P)=0$. Let $S$ be a finite set of places in $\Omega$.
We have a commutative diagram where the rows are exact and the columns are complexes.
\begin{center}
\begin{tikzcd}
H\times_K P(K) \arrow[r] \arrow[d]                 & G(K) \arrow[r] \arrow[d]                 & {H^1(K,F)} \arrow[r] \arrow[d]                 & 1 \\
\prod_{v\in S}H\times_K P(K_v) \arrow[d] \arrow[r] & \prod_{v\in S}G(K_v) \arrow[d] \arrow[r] & {\prod_{v\in S}H^1(K_v,F)} \arrow[d] \arrow[r] & 1 \\
 \Be_S(H\times P)^D\arrow[r]                  & \Be_S(G)^D \arrow[r]                     & \Sha^1_S(\widehat F)^D \arrow[r]               &1
\end{tikzcd}
\end{center}
Recall that we have an isomorphism $\Be_S(H\times P)\simeq \Sha^2_S(\widehat{P})$.
The exactness of the last row comes from the fact that $H^1(K,\widehat F)\simeq\Ker(\Br_{\text{a}}(G)\rightarrow\Br_{\text{a}}(H\times_K P))$ and $\Br_{\text{a}}(H\times_K P)\simeq H^2(K,\widehat P)$ (Cor 7.4 and Lemme 6.9 of \cite{sansuc1981groupe}). The commutativity of the right-bottom square can be deduced from Lemme 8.11 of \cite{sansuc1981groupe}.

Since $G$ also admits a flasque resolution, with the same proof as in Corollary 9.9 of \cite{Colliot_Th_l_ne_2015}, we have the following result taking the limit on $S$:
$$1\rightarrow\overline{G(K)}\rightarrow G(K_\Omega)\rightarrow\Be_\omega(G)^D$$
\end{proof}

\begin{lemma}
The following sequence is exact
$$1\rightarrow \overline{G(K)}\rightarrow\prod_{v\in \Omega}G(K_v)\rightarrow \Be_\omega(G)^D$$
where $G$ is a connected linear group over $K$.
\end{lemma}
\begin{proof}
We use the resolution 
$1\rightarrow G^{\text u}\rightarrow G\rightarrow G^{\text{red}}\rightarrow 1$ and the induced diagram
\begin{center}

\begin{tikzcd}
G^{\text u}(K) \arrow[r] \arrow[d] & G(K) \arrow[r] \arrow[d]            & G^{\text{red}}(K) \arrow[r] \arrow[d]        & 1 \\
G^{\text u}(K_\Omega) \arrow[r]    & G(K_\Omega) \arrow[r] \arrow[d]     & G^{\text{red}}(K_\Omega) \arrow[d] \\
                                   & \Be_\omega(G)^D \arrow[r] & \Be_{\omega}(G^{\text{red}})^D.               &  
\end{tikzcd}
\end{center}
taking into account the vanishing of $H^1(K,G^{\text u})$. The unipotent radical $G^{\text{u}}$ satisfies weak approximation (see Section \ref{notation}). The exactness of the rightmost column is known from the previous lemma, then a diagram chasing gives the result.
\end{proof}

\begin{lemma}
The sequence
$G(K_\Omega)\rightarrow\Be_\omega(G)^D\rightarrow\Be(G)^D$ is exact, where $G$ is a connected linear group over $K$.
\end{lemma}
\begin{proof}
As in the previous lemmas, we can first treat the case where $G$ is reductive. Suppose now that $G$ is reductive, we use a coflasque resolution $$1\rightarrow P\rightarrow G^\prime\rightarrow G\rightarrow 1$$ where $P$ is a quasitrivial torus and $G^\prime$ fits into the exact sequence
$$1\rightarrow G^{\text{sc}}\rightarrow G^\prime\rightarrow T\rightarrow 1$$ where $T$ is a coflasque torus and $G^{\text{sc}}$ is a semisimple simply connected group.
Suppose that the exactness of the sequence in question is known for $G^\prime$ (replacing $G$). Then we can prove the exactness for $G$ by chasing in the following diagram
\begin{center}
\begin{tikzcd}
                                              & \prod_{v\in S}G^\prime(K_v) \arrow[d] \arrow[r] & \prod_{v\in S}G(K_v) \arrow[d] \arrow[r] & 0 \\
\Be_\omega(P)^D \arrow[r] \arrow[d, "\simeq"] & \Be_\omega(G^\prime)^D \arrow[r] \arrow[d]      & \Be_\omega(G)^D \arrow[r] \arrow[d]      & 0 \\
\Be(P)^D \arrow[r]                            & \Be(G^\prime)^D \arrow[r]                       & \Be(G)^D \arrow[r]                       & 0.
\end{tikzcd}\end{center}
The leftmost arrow is an isomorphism by Proposition 2.6 of \cite{Colliot_Th_l_ne_2015}.
The last two rows are exact, and this can be proved by chasing the following diagram, as done in the proof of Lemma 4.4 of \cite{borovoi1996brauer}.

\begin{center}
    \begin{tikzcd}
            &                                            & \Br(K) \arrow[ld, "f_1" description] \arrow[d, "f_2" description]         &                           \\
0 \arrow[r] & \Br_1(G) \arrow[r, "f_4"] \arrow[d, "f_6"] & \Br_1(G^\prime) \arrow[r, "f_5"] \arrow[d, "f_7"]                                                          & \Br_{\text{a}}(P) \arrow[d, "f_3"] \\
0 \arrow[r] & \Br_1(G_v) \arrow[r, "f_9"]                & \Br_1(G_v^\prime) \arrow[r, "f_{10}"]                                                                         & \Br_{\text{a}}(P_v)                \\
            &                                            & \Br(K_v) \arrow[lu, "f_{11}" description] \arrow[u, "f_{8}" description] &                          
\end{tikzcd}
\end{center}
(The exactness of the rows can be deduced from Proposition 6.10 of \cite{sansuc1981groupe} and the fact that $\Pic(P)=H^1(K,\widehat P)=0$ since $\widehat P$ is a permutation module.)

\par
The exactness of the sequence in question is known for $T$ (Corollary 9.9 of \cite{Colliot_Th_l_ne_2015} for $K$ of type (a), and Theorem 4.9 of \cite{izquierdo2019dualite} for $K$ of type (b)), by a similar diagram chasing, we can prove the result for $G^\prime$, thus completing the proof. 
The essential condition we need in this diagram chasing is that we should have an injection $\Be_\omega(G^\prime)^D\hookrightarrow\Be_\omega(T)^D$, or equivalently a surjection $\Be_\omega(T)\twoheadrightarrow\Be_\omega(G^\prime)$. 
This is indeed true because we have an exact sequence $\Br_{\text{a}}(T)\rightarrow \Br_{\text{a}}(G^\prime)\rightarrow \Br_{\text{a}}(G^{\text{sc}})$  and $\Br(G^{\text{sc}})=\Br(K)$ since $G^{\text{sc}}$ a simply connected group over a field of characteristic 0 (c.f. Corollary in \S 0 of \cite{gille2009brauer}). 
\par
Now for a connected linear group $G$ (not necessarily reductive), we use again the resolution 
$1\rightarrow G^{\text u}\rightarrow G\rightarrow G^{\text{red}}\rightarrow 1$ and the induced commutative diagram 
\begin{center}
\begin{tikzcd}
& \prod_{v\in S}G(K_v) \arrow[d] \arrow[r] & \prod_{v\in S}G^{\text{red}}(K_v) \arrow[d] \arrow[r] & 1 \\
1 \arrow[r] & \Be_\omega(G)^D \arrow[r] \arrow[d]      & \Be_\omega(G^{\text{red}})^D \arrow[d]      &  \\
& \Be(G)^D \arrow[r]                       & \Be(G^{\text{red}})^D       & 
\end{tikzcd}\end{center}
taking into account the vanishing of $H^1(K_v,G^{\text u})$. The injectivity of the middle line follows from the fact that $\Br_{\text a}(G^{\text u})=\Br_{\text a}(\A_K^n)=(\Br_{\text a}(\A_K^1))^n=0$ (c.f. Lemme 6.6 of \cite{sansuc1981groupe}, Theorem 4.5.1(viii) of \cite{colliot2021brauer}). Then a diagram chasing gives the desired result.
\end{proof}

\section{Weak approximation for homogeneous spaces}
Now we look at weak approximation for homogeneous spaces over such fields $K$. The usual Brauer-Manin obstruction we used in last section is not enough in this case, as shown by the following construction of examples:
\begin{prop}\label{thm2}
Let $Q$ be a flasque torus such that
$$H^1(K,Q)\rightarrow\oplus_{v\in\Omega}H^1(K_v,Q)\rightarrow\Sha^1_{\omega}(\widehat Q)^D$$ is not exact. Such a torus exists (see Corollary 9.16 of \cite{Colliot_Th_l_ne_2015} for a construction.) 
Embed $Q$ into some $\SL_n$ and let $Z:=\SL_n/Q$. 
Then $\overline {Z(K)}\subsetneq Z(K_\Omega)^{\Be_\omega}.$
\end{prop}
\begin{proof}
We consider the following commutative diagram with exact rows:
\begin{center}
\begin{tikzcd}
                 & Z(K) \arrow[r,"f_1"] \arrow[d,"f_2"]                 & {H^1(K,Q)} \arrow[r,] \arrow[d,"f_3"]                 & 1 \\
 & Z(K_\Omega)\arrow[d,"f_5"]  \arrow[r,"f_4"] & {\prod_{v\in \Omega}H^1(K_v,Q)} \arrow[d,"f_6"] \arrow[r] & 1 \\ 1\arrow[r]&
 \Be_\omega(Z) \arrow[r,"f_{7}"]                     &  \Sha^1_\omega(\widehat Q)^D\arrow[r]& 1.
\end{tikzcd}
\end{center}
The vanishing of $\Br_{\text{a}}(\SL_n)$ on the bottom-left corner comes from the Corollary of \cite{gille2009brauer} (section 0). We prove the result by diagram chasing. Since the right column is not exact, there exists
$a\in \prod_{v\in \Omega}H^1(K_v,Q)$ such that $f_6(a)$ vanishes but $a\notin f_3(H^1(K,Q)).$ Since $f_4$ is surjective, we can find $b\in Z(K_\Omega)$ such that $f_4(b)=a$. By the commutativity of the bottom square, $f_5(b)$ vanishes. We prove that $b\notin \overline{Z(K)}$ by contradiction. We suppose $b\in \overline {Z(K)}$. The torus $Q$ being flasque implies that $\prod_{v\in \Omega}H^1(K_v,Q)$ is finite (c.f. Proposition 9.1 of \cite{Colliot_Th_l_ne_2015}), thus the preimage $f_4^{-1}(a)$ is open ($f_4$ is continuous, c.f. \cite{vcesnavivcius2015topology}), containing $b\in \overline{Z(K)}$, so we should be able to find $c\in Z(K)$ lying in $f_4^{-1}(a)$. Then $f_3(f_1(c))=a$ by the commutativity of the top square, contradicting $a\notin \im f_3$. Therefore, we found $b\in Z(K_\Omega)^{\Be_\omega}\backslash\overline{Z(K)}.$
\end{proof}

Therefore, we need to find other obstructions. In the rest of this section, we consider homogeneous spaces $Z$ under a connected linear group $G$ with geometric stabilizer $\bar H$ such that $G^{\text{ss}}$ is simply connected and $\bar H^{\text{torf}}$ is abelian. This is the assumption already used in \cite{izquierdo2021local} and \cite{borovoi1996brauer}, and is satisfied by every homogeneous space under a connected linear $K$-group with connected stabilizers (c.f. Lemma 5.2 in \cite{borovoi1996brauer}).
\par

We can find $Z\leftarrow W\rightarrow W^\prime$ such that
\begin{itemize}
    \item $W$ is a $K$-homogeneous space under $G\times T,$ 
    \item $T$ is a quasitrivial torus into which $H^{\text{torf}}$ injects, where $H^{\text{torf}}$ is the canonical $K$-form of $\bar H^{\text{torf}}$ associated to $Z$,
    \item $W\rightarrow Z$ is a $T$-torsor,
    \item $W^\prime$ is the quotient variety $Z/G^{\text{ss}}$, which is also a homogeneous space of $G^{\text{tor}}\times T$ with geometric stabilizer $\bar H^{\text{torf}}$ and the fibers of $W\rightarrow W^\prime$ are homogeneous spaces of $G^{\text{ssu}}$ with geometric stabilizers $\bar H^{\text{ssu}}$.
\end{itemize}
Indeed, we can embed $H^{\text{tor}}$ in a quasi-trivial torus $T$ and consider the diagonal morphism $H\rightarrow G\times T$ induced by the inclusion $H\hookrightarrow G$ and the composition $H\rightarrow H^{\text{tor}}\rightarrow T$. Then we define $W=(G\times T)/H$, and $W\rightarrow Z$ is induced by the projection to the first coordinate. Define $W^{\prime}$ to be the quotient variety $W/G^{\text{ssu}}$ and we get what we want.
\begin{prop}\label{wa}
The fiber $W_P$ above a $K$-point $P\in W^\prime(K)$ satisfies weak approximation.
\end{prop}

\begin{proof}

Since $H^{\text{ss}}$ is semi-simple, we consider its simply connected covering $1\rightarrow F\rightarrow H^{\text{sc}}\rightarrow H^{\text{ss}}\rightarrow 1$ where $F$ is finite and $H^{\text{sc}}$ is simply connected. 
\par
We first prove that this covering induces isomorphisms $H^1(K,H^{\text{ss}})\simeq H^2(K,F)$ and $H^1(K_v,H^{\text{ss}})\simeq H^2(K_v,F)$ for all $v\in\Omega$. For $K$ and $K_v$, the two conditions in Theorem 2.1 of \cite{colliot2004arithmetic} are satisfied, and thus we have a bijection $H^1(K,H^\text{ad})\rightarrow H^2(K,\mu)$ coming from the central isogeny $$1\rightarrow\mu\rightarrow H^{\text{sc}}\rightarrow H^\text{ad}\rightarrow1$$ associated to the center $\mu$ of $H^{\text{sc}}$. Since $F$ is contained in the center $\mu$, we have the exact sequence $$1\rightarrow \mu/F\rightarrow H^{\text{ss}}\rightarrow H^\text{ad}\rightarrow1$$ and the commutative diagram with exact rows
\begin{center}
    \begin{tikzcd}
{H^1(K,\mu/F)} \arrow[d, no head, double] \arrow[r] & {H^1(K,H^{\text{ss}})} \arrow[d] \arrow[r] & {H^1(K,H^\text{ad})} \arrow[d, no head, double] \arrow[r] & {H^2(K,\mu/F)} \arrow[d, no head, double] \\
{H^1(K,\mu/F)} \arrow[r]                                & {H^2(K,F)} \arrow[r]                & {H^2(K,\mu)} \arrow[r]                                   & {H^2(K,\mu/F).}                               
\end{tikzcd}
\end{center}
Then the four-lemma gives the surjectivity of $H^1(K,H^{\text{ss}})\rightarrow H^2(K,F)$. The injectivity comes from the vanishing of $H^1(K,H^{\text{sc}})$. The proof is the same when $K$ is replaced by $K_v$.

\par Then, we  prove that $$H^2(K,F)\rightarrow \oplus_{v\in S}H^2(K_v,F)$$ is surjective, where $S$ is a finite set of places. 
This is proved by chasing the following commutative diagram
\begin{center}
    \begin{tikzcd}
                               & {\oplus_{v\notin S}H^2(K_v,F)} \arrow[d] \arrow[rd,two heads] &                             &   \\
{H^2(K,F)} \arrow[r] \arrow[d] & {\oplus_{v\in\Omega}H^2(K_v,F)} \arrow[r] \arrow[d] & {H^0(K,\widehat F)^D} \arrow[r] & 1. \\
{H^2(K,F)} \arrow[r]           & {\oplus_{v\in S}H^2(K_v,F)} \arrow[d]               &                             &   \\
                               & 1                                                   &                             &  
\end{tikzcd}
\end{center}
The row in the middle is exact (c.f. \cite{izquierdo2016theoremes} Theorem 2.7 applying $d=0$. The surjectivity of $\oplus_{v\notin S}H^2(K_v,F)\rightarrow H^0(K,\hat F)^D$ comes from the injectivity taking the duals $H^0(K,\hat F)\hookrightarrow\oplus_{v\notin S} H^0(K_v,\hat F)$. This yields the surjectivity of $H^2(K,F)\rightarrow \prod_{v\in S}H^2(K_v,F)$,  which is $H^1(K,H^{\text{ss}})\rightarrow \oplus_{v\in S}H^1(K_v,H ^{\text{ss}})$.
\par
Finally, by Lemme 1.13 of \cite{sansuc1981groupe}, we have $H^1(K,H^{\text{ss}})=H^1(K,H^{\text{ssu}})$ and $H^1(K_v,H^{\text{ss}})=H^1(K_v,H^{\text{ssu}})$. With the same argument as in Proposition 3.2 of \cite{sansuc1981groupe}, we can prove that $G^{\text{ssu}}$ satisfies weak approximation (c.f. Proposition 4.1 of \cite{borel1966algebraic} for the vanishing of $H^1(K,H^u)$ and $H^1(K_v,H^u)$. The unipotent radical $H^{\text{u}}$ satisfies weak approximation).
Therefore, we consider the commutative diagram with exact rows (the set $W_P(K)$ is non-empty by Proposition 3.1 of \cite{izquierdo2021local}):
\begin{center}
    \begin{tikzcd}
G^{\text{ssu}}(K) \arrow[r] \arrow[d]       & W_P(K) \arrow[r] \arrow[d]       & {H^1(K,H^{\text{ssu}})} \arrow[r] \arrow[d, two heads] & 1 \\
\prod_{v\in S}G^{\text{ssu}}(K_v) \arrow[r] & \prod_{v\in S}W_P(K_v) \arrow[r] & {\prod_{v\in S}H^1(K_v,H^{\text{ssu}})} \arrow[r]         & 1.
\end{tikzcd}
\end{center}
A diagram chasing then gives the desired result. 
\end{proof}
\begin{prop}
We have $W(K_{\Omega})^{\Be_\omega(W)}=\overline{W(K)}$.
\end{prop}
\begin{proof}
Consider $(P_v)\in W(K_{\Omega})^{\Be_\omega(W)}$, and any open subset $U$ containing $(P_v)$, we want to find $Q\in U\cap W(K)$ using fibration methods.
\par Denote by $(P_v^\prime)$ the image of $(P_v)$ under the induced map $g: W(K_\Omega)\rightarrow W^\prime(K_\Omega)$. By the functorality of the Brauer-Manin pairing, $(P_v^\prime)$ lies in $W^\prime(K_{\Omega})^{\Be_\omega(W^\prime)}$. Since $W\rightarrow W^
\prime$ is smooth, and the $K_v$ are henselian, we have open maps $W(K_v)\rightarrow W^\prime(K_v)$, and thus $g$ is open. In particular, $V=g(U)\subseteq W^\prime(K_\Omega)$ is also open. Since $W'$ is a torus (c.f. Theorem 3.2 of \cite{izquierdo2021local}), we have $W^\prime(K_{\Omega})^{\Be_\omega(W^\prime)}=\overline{W^\prime(K)}$. Therefore, there exists $P^\prime\in V\cap W^\prime(K)$. Let $(Q_v)\in g^{-1}(P^{\prime})\cap U\subseteq W(K_\Omega)$, and it can be seen as in $ W_{P^\prime }(K_\Omega)$ too, as shown in the following diagram.
\begin{center}
    \begin{tikzcd}
                                                                                   & W \arrow[r] \arrow[rd, phantom,"\square"] & W^\prime                       \\
                                                                                   & W_{P^\prime} \arrow[u] \arrow[r]         & \Spec K \arrow[u, "P^\prime "] \\
\coprod_{v\in\Omega}\Spec K_v \arrow[ruu, "(Q_v)"'] \arrow[rru] \arrow[ru, dotted] &                                 &                               
\end{tikzcd}
\end{center}
This diagram also shows that $W_{P^\prime}(K_\Omega)\subseteq W(K_\Omega)$, and $U_{|_{ W_{P^\prime}(K_\Omega)}}:=U\bigcap W_{P^\prime}(K_\Omega)$ is an open neighborhood of $(Q_v)$ in $W_{P^\prime}(K_\Omega)$. Since $W_{P^\prime}$ satisfies weak approximation by proposition \ref{wa}, there exists $Q\in W_{P^\prime}(K)\cap U_{|_{ W_{P^\prime}(K_\Omega)}}$, and thus in $W(K)\cap U$, which proves that $W(K_{\Omega})^{\Be_\omega(W)}\subseteq\overline{W(K)}$.
\end{proof}
As in \cite{izquierdo2021local}, we are naturally led to consider the following definition.
\begin{defn}
For an arbitrary $K$-variety $Z$, we define $$Z(K_\Omega)^{\text{\text{qt}},\Be_\omega}:=\bigcap_{f:W\xrightarrow T{Z},T\text{ quasi-trivial}}f(W(K_\Omega)^{\Be_\omega}).$$
\end{defn} Using the torsor $W\rightarrow Z$ defined above, we have $Z(K_\Omega)^{\text{qt},\Be_\omega}\subseteq\overline{Z(K)}$. In fact, this is an equality.
\begin{prop}\label{thm3}
Let the notation be as above. Then $Z(K_\Omega)^{\text{qt},\Be_\omega}=\overline{Z(K)}.$
\end{prop}
\begin{proof}
Since for all $f$, we have $$Z(K)\subseteq f(W(K))\subseteq f(W(K_\Omega)^{\Be_\omega}),$$ so $Z(K)\subseteq Z(K_\Omega)^{\text{qt},\Be_\omega}$. To prove Proposition \ref{thm3}, it is equivalent to prove that $Z(K_\Omega)^{\text{qt},\Be_\omega}$ is closed. It suffices to prove that $f(W(K_\Omega)^{\Be_\omega})$ is closed for every torus $f: W\xrightarrow T{Z}$ with $T$ quasitrivial.
Since $f$ is smooth and the $K_v$ are henselian, the induced map $f:W(K_\Omega)\rightarrow Z(K_\Omega)$ is open. Now we'll prove that $f(W(K_\Omega)^{\Be_\omega})$ doesn't meet $f(W(K_\Omega)\backslash W(K_\Omega)^{\Be_\omega})$, and combined with the surjectivity of $f$, we'll get $f(W(K_\Omega)^{\Be_\omega})$ closed.
\par
 Fixing $w\in W(K)$ which is a $K$-point of $W$, we have the canonical isomorphism $\Br_{\text{a}}(W)\simeq \Br_{1,w}(W)$. Define the map $i_w: T\rightarrow W$ as $t\mapsto t.w$, and we have an induced map $i_w^*: \Br_{1,w}(W)\rightarrow\Br_{1,e}(T)$. 
Let $\alpha$ be an element in $\Be_\omega(W)$ and we still denote by $\alpha$ its image in $\Br_{1,w}(W)$. Let $(x_v),(y_v)\in W(K_\Omega)$ above $(P_v)$. We'll prove that $\sum_{v\in\Omega}<x_v,\alpha>=\sum_{v\in\Omega}<y_v,\alpha>$. Let $m: T\times W\rightarrow W$ denotes the action of $T$ on $W$. 
Since $W_{P_v}$ is a $K_v$-torsor under $T$, there exists $t_v\in T(K_v)$ such that $x_v=t_v.y_v$. By the functorality of the Brauer-Manin paring, we have $<x_v,\alpha>=<(t_v,y_v),m^* \alpha>$. By Lemme 6.6 of \cite{sansuc1981groupe}, we have $\Br_{1,(w,e)}(W\times T)=\Br_{1,w}(W)\times\Br_{1,e}(T)$, and the formula (24) in \cite{borovoi2013manin} gives $m^*\alpha=p_T^*i_w^*(\alpha)+p_W^*\alpha$ where $p_T$ and $p_W$ are the two natural projections. Then $<x_v,\alpha>=<t_v,i_w^*(\alpha)>+<y_v,\alpha>$. But $P$ is a quasi-trivial torus, we have $\Be_\omega(P)\simeq \Be(P)$ (c.f. Prop. 2.6 and \S 8.1 of \cite{Colliot_Th_l_ne_2015}). Therefore, by the exact sequence in Theorem \ref{4.1}, the map $P(K_\Omega)\rightarrow \Be_\omega(P)^D$ is $0$, i.e. $\sum_{v\in\Omega}<t_v,i_w^*(\alpha)>=0$ and we get what we want.
\end{proof}
\begin{remark}
Actually the above proof also shows that \begin{equation*}
Z(K_{\Omega})^{\text{qt},\Be_\omega}=\bigcap_{f:W\xrightarrow T{Z},T\text{ quasi-trivial}}\bigcap_{\alpha\in{\Be_\omega}(W)}f(W(K_\Omega)^{\alpha}).
\end{equation*} 
 With this description, we can prove as done in Theorem 6.4 of \cite{izquierdo2021local} that 
\begin{equation}\label{inclusion}Z(K_\Omega)^{\text{tor}}\subseteq Z(K_\Omega)^{\text{qt},\Be_\omega}.\end{equation} But this doesn't necessarily give an obstruction to weak approximation since we don't know if $Z(K_\Omega)^{\text{tor}}$ is closed. We don't know if (\ref{inclusion}) is an equality. One could wonder whether a ``purely descent'' description of $Z(K_\Omega)^{\text{qt},\Be_\omega}$ exists.
\end{remark}
\bibliographystyle{alpha}
\bibliography{bib}

\begin{thebibliography}{CTGP04}

\bibitem[BD13]{borovoi2013manin}
Mikhail Borovoi and Cyril Demarche.
\newblock Manin obstruction to strong approximation for homogeneous spaces.
\newblock {\em Commentarii Mathematici Helvetici}, 88(1):1--54, 2013.

\bibitem[Bor66]{borel1966algebraic}
Armand Borel.
\newblock Algebraic groups and discontinuous subgroups.
\newblock In {\em Proc. Sympos. Pure Math.} Amer. Math. Soc., Providence, 1966.

\bibitem[Bor96]{borovoi1996brauer}
Mikhail Borovoi.
\newblock The brauer-manin obstructions for homogeneous spaces with connected
  or abelian stabilizer.
\newblock {\em Journal f{\"u}r die reine und angewandte Mathematik (Crelles
  Journal)}, 1996(473):181--194, 1996.

\bibitem[{\v{C}}es15]{vcesnavivcius2015topology}
K{\k{e}}stutis {\v{C}}esnavi{\v{c}}ius.
\newblock Topology on cohomology of local fields.
\newblock In {\em Forum of Mathematics, Sigma}, volume~3. Cambridge University
  Press, 2015.

\bibitem[CTGP04]{colliot2004arithmetic}
Jean-Louis Colliot-Th{\'e}l{\`e}ne, Philippe Gille, and Raman Parimala.
\newblock Arithmetic of linear algebraic groups over 2-dimensional geometric
  fields.
\newblock {\em Duke Mathematical Journal}, 121(2):285--342, 2004.

\bibitem[CTH15]{Colliot_Th_l_ne_2015}
Jean-Louis Colliot-Thélène and David Harari.
\newblock Dualité et principe local-global pour les tores sur une courbe
  au-dessus de ℂ((t )).
\newblock {\em Proceedings of the London Mathematical Society},
  110(6):1475–1516, Apr 2015.

\bibitem[CTS21]{colliot2021brauer}
Jean-Louis Colliot-Th{\'e}lene and Alexei~N Skorobogatov.
\newblock {\em The Brauer--Grothendieck group}, volume~71.
\newblock Springer Nature, 2021.

\bibitem[Gil09]{gille2009brauer}
Stefan Gille.
\newblock On the brauer group of a semisimple algebraic group.
\newblock {\em Advances in Mathematics}, 220(3):913--925, 2009.

\bibitem[IA21]{izquierdo2021local}
Diego Izquierdo and Giancarlo~Lucchini Arteche.
\newblock Local-global principles for homogeneous spaces over some
  two-dimensional geometric global fields.
\newblock {\em Journal f{\"u}r die reine und angewandte Mathematik (Crelles
  Journal)}, 2021.

\bibitem[Izq16]{izquierdo2016theoremes}
Diego Izquierdo.
\newblock Th{\'e}or{\`e}mes de dualit{\'e} pour les corps de fonctions sur des
  corps locaux sup{\'e}rieurs.
\newblock {\em Mathematische Zeitschrift}, 284(1):615--642, 2016.

\bibitem[Izq19]{izquierdo2019dualite}
Diego Izquierdo.
\newblock Dualit{\'e} et principe local-global pour les anneaux locaux
  hens{\'e}liens de dimension 2.
\newblock {\em With an appendix by Jo{\"e}l Riou. Algebr. Geom}, 6(2):148--176,
  2019.

\bibitem[Poo17]{poonen2017rational}
Bjorn Poonen.
\newblock {\em Rational points on varieties}, volume 186.
\newblock American Mathematical Soc., 2017.

\bibitem[San81]{sansuc1981groupe}
J-J Sansuc.
\newblock Groupe de brauer et arithm{\'e}tique des groupes alg{\'e}briques
  lin{\'e}aires sur un corps de nombres.
\newblock 1981.

\end{thebibliography}
\end{document}